\documentclass[leqno,11pt]{amsart}
\usepackage[colorlinks,pagebackref,hypertexnames=false]{hyperref}

\usepackage{amsmath,amsthm,amssymb,mathrsfs} 
\usepackage[alphabetic,backrefs]{amsrefs}
\usepackage{ae,aecompl}
\usepackage[margin=1in]{geometry} 
\usepackage[retainorgcmds]{IEEEtrantools}
\usepackage{soul}

\usepackage[english]{babel}

\usepackage{bookmark}

\usepackage{xypic}
\usepackage{amsmath,amsfonts,amssymb,graphics,graphicx,latexsym,amscd,subfigure,hyperref}
\usepackage[usenames]{xcolor}
\usepackage[retainorgcmds]{IEEEtrantools}
\usepackage{graphicx}
\usepackage{amssymb}
\usepackage{xypic}
\usepackage{hyperref}
\usepackage{psfrag}
\usepackage{amsmath,amssymb}
\usepackage{dsfont}

\newtheorem{thm}{Theorem}[section]

\newtheorem{cor}[thm]{Corollary}
\newtheorem{prop}[thm]{Proposition}

\newtheorem{lemma}[thm]{Lemma}
\newtheorem{defn}[thm]{Definition}
\newtheorem{rem}[thm]{Remark}

\newcommand{\bbR}{\mathbb{R}}

\newcommand{\bbT}{\mathbb{T}}
\newcommand{\bbC}{\mathbb{C}}

\newcommand{\bbZ}{\mathbb{Z}}
\newcommand{\bbH}{\mathbb{H}}

\newcommand{\grad}{\triangledown}

\newcommand{\mP}{{\mathcal P}}
\newcommand{\mA}{\mathcal{A}}
\newcommand{\mB}{{\mathcal B}}

\newcommand{\bz}{{\bf z}}
\newcommand{\bx}{{\bf x}}
\newcommand{\fut}{{\eta}}

\newcommand{\CP}{\mathbb C \mathbb P}

\DeclareMathOperator{\Lap}{\triangle}

\DeclareMathOperator{\spec}{Spec}

\DeclareMathOperator{\Hess}{Hess}

\DeclareMathOperator{\scal}{Scal}

\DeclareMathOperator{\interior}{Int}
\DeclareMathOperator{\tr}{Tr}

\newcommand{\cp}{{\mathbb{CP}}}

\begin{document}
\title[Hermitian, Ricci-flat, toric metrics]{Hermitian, Ricci-flat toric metrics on non-compact surfaces \`a la Biquard-Gauduchon}
\author{Gon\c calo Oliveira, Rosa Sena-Dias}
\begin{thanks}
{Partially supported by the Funda\c{c}\~{a}o para a Ci\^{e}ncia e a Tecnologia (FCT/Portugal) through project EXPL/MAT-PUR/1408/2021EKsta.}
\end{thanks}

\address{Centro de An\'{a}lise Matem\'{a}tica, Geometria e Sistemas Din\^{a}micos, Departamento de Matem\'atica, Instituto Superior T\'ecnico}
\email{galato97@gmail.com, rsenadias@math.ist.utl.pt}

\begin{abstract} 
In \cite{bg}, Biquard-Gauduchon show that conformally K\"ahler, Ricci-flat, ALF toric metrics on the complement of toric divisors are: the Taub-NUT metric with reversed orientation, in the Kerr-Taub-bolt family or in the Chen-Teo family. In the same paper, Biquard-Gauduchon also give a unified construction for the above families relying on an axi-symmetric harmonic function on $\bbR^3$. In this work, we reverse this construction and use methods from \cite{s} to show that {\it all} conformally K\"ahler, Ricci-flat, toric metrics on the complement of toric divisors, under some mild assumptions on the associated moment polytope, are among the families above. In particular all such metrics are ALF.
\end{abstract}

\maketitle 
\tableofcontents
\section{Introduction}
Einstein metrics play an important role in geometry. In the K\"ahler setting a special subclass of such metrics has been much studied, namely K\"ahler-Einstein metrics. The existence problem for K\"ahler-Einstein metrics in the first Chern class of a K\"ahler manifold has been the subject of intense investigation. Chen-Donaldson-Sun (see \cite{cds} and subsequent articles by the same authors) have proved that such metrics exist if and only if the underlying K\"ahler class is  K-stable. In particular, there are K\"ahler $4$-manifolds such as $\CP^2\sharp\overline{\CP}^2$ or $\CP^2\sharp2\overline{\CP}^2$ which do not admit any K\"ahler-Einstein metric. It is then natural to look for Einstein metrics that are Hermitian i.e. for which the metric is preserved by the complex structure.  Page constructed one such metric on $\CP^2\sharp\overline{\CP}^2$. 

It follows from the Riemannian Goldberg-Sachs Theorem and the work of Derdzinski that, in dimension four, all Hermitian-Einstein metrics are actually conformally K\"ahler \cite{l0,AG}. Derdzinski further showed that, under certain assumptions, extremal, Bach-flat, K\"ahler metrics on $4$-manifolds admit an Einstein metric in their conformal classes. Using Derdzinski's results, Chen-LeBrun-Weber (see \cite{clw}) proved the existence of a Hermitian-Einstein metric on $\CP^2\sharp2\overline{\CP}^2.$ It is natural to wonder wether Derdzinski's results could yield new Einstein metrics in the non-compact case and to what extent such metrics are unique when they exist. 

In \cite{bg}, Biquard and Gauduchon applied Derdzinski's methods to construct Hermitian, Ricci-flat, toric metrics on non-compact spaces having a Poincar\'e type behaviour at infinity. In fact, all of the metrics they construct this way had been previously written down. The metrics are: the Taub-NUT metric with reversed orientation on the complement of a divisor in $\cp^2$; in the Kerr family on the complement of a divisor in $\CP^2\sharp\overline{\CP}^2$; in the Taub-bolt family in the complement of a different divisor in $\CP^2\sharp\overline{\CP}^2$, or in the Chen-Teo family in the complement of a divisor in $\CP^2\sharp2\overline{\CP}^2.$ Although the smooth metrics appearing in \cite{bg} are not new, the authors give a new, unified construction for all the metrics in the above families. They also prove a classification result for Hermitian, Ricci-flat, toric metrics with ALF behaviour at infinity. 
\begin{thm}[Biquard-Gauduchon]\label{thmBG}
Let $(X,g)$ be a Hermitian, non-K\"ahler, Ricci-flat, ALF toric metric. Then $g$ is either one of the metrics constructed by Chen-Teo, one of the Kerr-Taub-bolt metrics or the Taub-NUT metric with opposite orientation.
\end{thm}
\vspace{-0.2in}
\begin{figure}[h]
	\centering
	\includegraphics[scale=0.45]{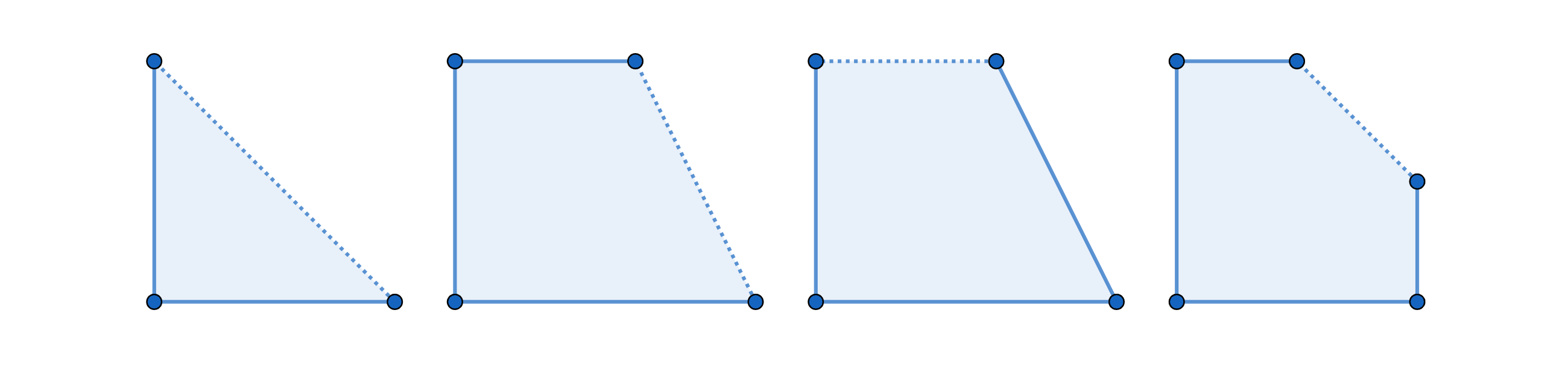}
	\caption{From left to right, polytopes corresponding to the reversed Taub-Nut, Kerr, Taub-bolt and Chen--Teo families.}
\end{figure}

The Taub-NUT metric is a K\"ahler, Ricci-flat, toric metric on $\bbR^4$ which is also the complement of a torus invariant divisor in $\bbC^2$. There is a different complex structure in $\bbR^4$ with respect to which the Taub-NUT metric is no longer K\"ahler but rather Hermitian. It is asymptotic to $\bbR^2\times S^1$ in a precise manner at infinity, hence ALF. The Kerr-Taub-bolt metrics are other examples of K\"ahler Ricci-flat metric with a $\bbT^2$-symmetry. They are defined on the complement of a divisor in $\CP^2\sharp\overline{\CP}^2$ and include the famous Schwarzschild metric as a degeneration. 
Many compact examples of Ricci-flat K\"ahler metrics arise via hyperK\"ahler geometry. Ricci-flat, complete, non-compact metrics, with a certain assumption on the decay of the curvature at infinity, also known as instantons, have been well studied too, particularly K\"ahler gravitational instantons. We currently have classification result for such K\"ahler gravitational instantons thanks to Chen-Chen (see \cite{cc}). In particular their behaviour at infinity is known to be restricted. 

In \cite{s}, the second author obtained a classification result for non-compact, K\"ahler, scalar-flat, toric metrics. As it turns out, K\"ahler, scalar-flat, toric metrics all arise from a construction in \cite{as} which is based on an ansatz of Joyce that was translated into the language of action-angle coordinates on toric K\"ahler manifolds by Donaldson (see \cite{donjoyce}). Such metrics are known as Donaldson generalised Taub-NUT metrics.
\begin{thm}[\cite{s}]\label{thm_us}
Let $(X,J,g)$ be a strictly unbounded toric K\"ahler surface. If $g$ is scalar-flat, then $g$ is equivariantly isometric to a Donaldson generalised Taub-NUT metric.
\end{thm}
By using Joyce's ansatz, the theorem results from a classification of harmonic functions in Euclidean space with prescribed singularities. Toric geometry plays an important role and convexity is used to prove the positivity of an appropriate harmonic function. One of the main tools is a result due to Donaldson (see  \cite{donjoyce}) that any K\"ahler, scalar-flat, toric metric can be written via Joyce's ansatz.

The goal of this note is to apply the methods in \cite{s} to the setting of conformally K\"ahler, Ricci-flat, toric metrics. We shall use Biquard-Gauduchon's ansatz translated into the language of toric geometry in replacement of Joyce's ansatz. We will show that any conformally K\"ahler, Ricci-flat, toric metric can be written in terms of Biquard-Gauduchon's ansatz via a single axi-symmetric harmonic function on the upper half plane.
 
Our uniqueness result for such metrics is stronger than the result in \cite{bg} as it does not make use of the ALF hypothesis. In fact, contrary to most uniqueness results for special metrics on non-compact manifolds, our theorem holds without assumptions on asymptotics. To be more precise we prove the following.

\begin{thm}\label{thm_main}
Let $(M^4,J,\omega)$ be a compact toric K\"ahler surface. Let $D$ be a torus invariant divisor in $M.$ Let $X=M\setminus D$ and assume $X$ is endowed with two $J$-Hermitian, Ricci-flat,  {torus invariant} metrics $g$ and $\tilde{g}$ whose adapted K\"ahler forms\footnote{See the discussion following the statement of the Theorem, or Definition (\ref{def:adapted}) later in the text, for a precise definition of adapted K\"ahler form.} are cohomologous to $[\omega|_X]$. Then $g$ and $\tilde{g}$ are equivariantly isometric.  
\end{thm}
It follows from the Riemannian Goldberg-Sachs Theorem \cite{l0,AG} that such Hermitian and Einstein $4$-manifolds are necessarily conformally K\"ahler. Therefore, $g=e^{\sigma}g_K$ for a smooth function $\sigma$, which we can normalise to satisfy $\sigma(x)=0$ at some fixed $x \in X$, and a K\"ahler metric $g_K$ which together with $J$ determines a K\"ahler form $\omega_K$ which we call the \emph{adapted K\"ahler form}. Notice that this is uniquely determined by imposing that $\sigma$ vanishes at $x \in X$. In our theorem, we require that the adapted forms for the two conformally K\"ahler, Ricci-flat, toric metrics be in the same cohomology class. The  K\"ahler metric $g_K$ together with $J$ determines a symplectic form $\omega.$ If $J$ and $g$ are torus invariant then, as we will see, so is  $g_K$ and therefore $\omega.$ The torus action must be Hamiltonian for $\omega$ and it therefore defines a moment polytope which is of Delzant type. The cohomology class of the adapted form fixes a moment polytope for the torus action on $(M,J)$, up to $SL(2,\bbZ)$ transformations and translations, because by Moser's trick cohomologous equivariant forms are equivariantly symplectomorphic and therefore have the same moment polytope, again up to $SL(2,\bbZ)$ transformations and translations. Furthermore, we remark that torus invariant divisors $D$ as in the statement, correspond to moment map pre-images of codimension-$1$ faces of the moment polytope.

A canonical quantity which is associated with $(M^4,D,J,\omega)$ is the extremal vector field of $X=M\setminus D$ which can be thought of as a vector in $\bbR^2\simeq \left(\text{Lie}(\bbT^2)\right)^*$. This is essentially the same as the Futaki invariant of $X=M/D.$ We will give its precise definition ahead. The important fact for now is that it depends only on the symplectic structure and can be calculated using the Delzant polytope.

\begin{thm}\label{thm_main2}
{Let $(M^4,J,\omega)$ be a compact toric K\"ahler surface and $D$ a torus invariant divisor in $M.$ Assume that  $X=M\setminus D$ carries a $J$-Hermitian, Ricci-flat, toric metric whose adapted K\"ahler form is cohomologous to $[\omega|_X]$. Further, suppose that the extremal vector field of the adapted K\"ahler form is not normal to exactly one of the non-compact edges of the moment polytope of $X$. Then, the Hermitian metric on $X$ is equivariantly isometric to one of the Biquard-Gauduchon metrics from Theorem (\ref{thmBG})}.
\end{thm}

In particular, under the above assumptions all Hermitian, Ricci-flat, toric metrics are ALF. 
Recall that such metrics are conformally K\"ahler. Then, the adapted K\"ahler form $\omega_K$ lies in the class $[\omega |_X]$ for $X=M\setminus D$, where $(M,J,\omega)$ is compact K\"ahler and $D \subset M$ is a torus invariant divisor. In particular, the moment polytope of $(X,\omega_K)$ is obtained by removing an edge from the moment polytope of $(M,\omega)$. On the other hand, there are conformally K\"ahler, Ricci-flat, toric metrics which are not ALF. Consider the Eguchi-Hanson metric on the total space of the bundle $\mathcal{O}(-2)$ over $\mathbb{CP}^1$. It is K\"ahler with respect to the standard complex structure on $\mathcal{O}(-2).$ Now $\mathcal{O}(-2)$ admits a different complex structure yielding the reversed orientation, which is biholomorphic to $\mathcal{O}(2)$. The Eguchi-Hanson metric is Hermitian with respect to this complex structure. Now $\mathcal{O}(2)$ is biholomorphic to $\cp^2$ minus a point which is biholomorphic to $\bbH_1\setminus D$ for a torus invariant divisor $D$ of self-intersection $2.$ But the adapted K\"ahler form associated to this reversed Eguchi-Hanson metric determines a polytope which is a $2$-simplex minus a vertex. This is not the moment polytope of a manifold of the form $X=M\setminus D$ for $(M^4,J)$ toric and $D$ a torus invariant divisor in $M.$ The K\"ahler form associated to the conformally K\"ahler metric is not equivariantly symplectomorphic to any symplectic form on $M\setminus D.$ Another way to think of this is to note that if the Eguchi-Hanson metric were conformal to a K\"ahler metric on $\bbH_1\setminus D,$ the corresponding cohomology class would integrate to zero on the removed divisor which is not possible. 

Recently, Mingyang Li (see \cite{m}) obtained uniqueness results for conformally K\"ahler, Ricci-flat, ALE metrics. Elsewhere (see \cite{m2}), Li also generalised Biquard-Gauduchon's  classification result to all ALF, conformally K\"ahler, Ricci-flat metrics by showing that such metrics are necessarily toric. Our results only apply to the toric setting but we make no assumption on asymptotic growth. 

\section{Background}
\subsection{Conformally K\"ahler Einstein metrics}
Since we make use of several results from \cite{d}, we will start by briefly stating those in this subsection. For more details and proofs see \cite{d}.

Let $(X,g)$ be any Riemannian manifold. There is a splitting induced by the Hodge star corresponding to metric on the space $\Lambda^2T^*X.$ This lets us define self-dual and anti-self-dual forms i.e. the $\pm1$ eigenspaces of the Hodge star on $\Lambda^2T^*X.$ According to this splitting, the self-dual part of the Weyl curvature $W$ which we denote by $W^+$ is the part of the endomorphism $W$ of  $\Lambda^2T^*X$ which takes $\Lambda^{2+}T^*X,$ to itself.  Derdzinski interprets $W^+$ at a point $x\in X$ as an endomorphism of $\Lambda^{2+}T_x^*X,$ the space of self-dual sections of $\Lambda^2T_x^*X.$ At each $x\in X$, $\Lambda^{2+}T_x^*X,$ has dimension $3$ and $W^+$ has a real spectrum. 
\begin{defn}
The Riemannian manifold $(X,g)$ is said to be half-algebraically simple if at each point in $X,$ $W^+$ has at least one double eigenvalue. In this case we write $\sharp(\spec(W^+))\leq 2.$
\end{defn} 
When $X$ is a complex manifold we say that a metric $g$ on $X$ is Hermitian if $J$ preserves $g$ i.e.
$$
g_x(JV_x,JW_x)=g_x(V_x,W_x), \, \forall x\in \, X, \, V_x, W_x\in T_xX.
$$
When $(X,J, \omega, g_K)$ is K\"ahler, any metric conformal to $g_K$ is Hermitian. It follows from a Riemannian version of the Goldberg-Sachs Theorem \cite{l0,AG} that $4$-dimensional Hermitian Einstein metrics are also conformally K\"ahler.

We gather several results which are due to Derdzinski in a theorem below. 
\begin{thm}[Derdzinski]\label{Thmderd1}
Let $(M,\omega, g_K)$ be a K\"ahler manifold. 
\begin{itemize}
\item The Hermitian metric $\frac{g_K}{\scal^2(g_K)}$ is half-algebraically simple and satisfies $\delta W^+=0,$ where defined.
\item If $e^\sigma g_K$ satisfies $\delta W^+=0,$ then $e^{-\sigma}$ is a constant multiple of $\scal^2(g_K)$ on the subset where $W^+\ne 0.$
\item The Hermitian metric $\frac{g_K}{\scal^2(g_K)}$ is Einstein iff its Bach tensor vanishes identically.
\item If $\frac{g_K}{\scal^2(g_K)}$ is Einstein, then $\nabla\scal(g_K)$ is holomorphic and the metric $g_K$ is extremal in the sense of Calabi.
\end{itemize}
\end{thm}

Another result of Derdzinski's is the following.
\begin{thm}[Derdzinski]\label{Thmderd2}
Let $(X,g)$ be an oriented, half-algebraically simple, Einstein manifold. Let $\lambda$ be the only eigenvalue of $W^+$ which is not double and $\omega$ be the corresponding eingenform. Then $(X,\lambda^{2/3}g,\lambda^{2/3}\omega)$ is K\"ahler.
\end{thm}

Next, we introduce the following notion which is convenient in stating our results.

\begin{defn}\label{def:adapted}
	Let $(X,g,J)$ be conformally K\"ahler and $x \in X$. Then, $g=e^{\sigma}g_K$ for a smooth function $\sigma$ satisfying $\sigma(x)=0$ and a K\"ahler metric $g_K$ which together with $J$ determines a K\"ahler form $\omega_K$. In this situation, we shall say that $\omega_K$ is the K\"ahler form adapted to $g$ and $x$.
\end{defn}

\begin{rem}
	In order to have uniqueness in the definition of the adapted K\"ahler form we have fixed $\sigma(x)=0$ in the previous definition. Otherwise, the adapted K\"ahler form would only be well defined up to scaling.
\end{rem}

\subsection{Toric K\"ahler surfaces}
We give a very brief overview of toric K\"ahler geometry and action-angle coordinates mainly to set up notation. For more on this see \cite{a1}. 

A K\"ahler manifold $(X^4,\omega,g_K)$ is said to be toric if it admits an effective, isometric, Hamiltonian $\bbT^2$-action whose moment map is proper. Under such assumptions, the moment map 
$$
\bx=(x_1,x_2):X\rightarrow \bbR^2
$$
has as image a convex polytope $P$ of the form
$$
P=\{(x_1,x_2)\in \bbR^2: l_i(x_1,x_2):=\langle (x_1,x_2),\nu_i\rangle - \lambda_i\geq 0, i=1,\cdots, d+1\},
$$
where $\nu_i$ are the primitive, interior facet normals to $P.$ At each vertex of $P$ the facet normals corresponding to the facets meeting at that vertex say $\nu_j$ and $\nu_{j+1}$ for $j=1,\cdots, d$ satisfy
$$
\det(\nu_j,\nu_{j+1})=-1.
$$
The pre-image via $\bx$ of each facet of $P$ is a $\bbT^2$-invariant $\cp^1$ i.e. a divisor. 

On the pre-image of the interior of $P,$ the $\bbT^2$-action is free so that 
$$
\bx^{-1}(\interior(P))\simeq \interior(P)\times \bbT^2.
$$
Note that this set is dense in $X.$ Let $(\theta_1,\theta_2)$ denote angle coordinates on $\bbT^2. $ Then $(x_1,x_2,\theta_1,\theta_2)$ are the so-called action-angle coordinates coordinates on $\bx^{-1}(\interior(P))$. Action-angle coordinates are Darboux coordinates for $\omega$ i.e. 
$$
\omega=dx_1\wedge d\theta_1+dx_2\wedge d\theta_2.
$$
As for the metric $g_K,$ it is known that it can be expressed as 
$$
g_K=\sum_{i,j=1}^2 u_{ij}dx_i\otimes dx_j+u^{ij}d\theta_i\otimes d\theta_j
$$
where $u:\interior(P)\mapsto \bbR$ is a smooth convex function, $u_{ij}$ denote the entries of the Hessian of $u$ and $u^{ij}$ denote the entries of the inverse of the Hessian of $u.$ This is explained in \cite{s} for instance. The function $u$ is called the symplectic potential of the metric $g_K.$ The boundary behaviour of $u$ is determined by $P$. In particular,
$$
u-\sum_{i=1}^{d+1}l_i\log(l_i)
$$ 
is smooth up to the boundary of $P.$ Moreover, it has been shown that there is a smooth, positive function which we denote by $\delta$ such that,
$$
\det(\Hess u)=\frac{\delta}{\prod_{j=1}^{d+1}l_j}.
$$

We will also make use of the real sub-manifold of $X$ whose definition we quickly recall. For more details see \cite{s}.  Guillemin has shown that as complex manifolds, toric manifolds can be viewed as quotients of $\bbC^{d+1}$ by a subgroup of $(\bbC^*)^{d+1}$, where $\bbC^*$ acts on $\bbC$ in the usual way. Complex conjugation on $\bbC^{d+1}$ descends to $X$ and thus $X$ admits an isometric involution. The fixed point set of this involution is called the real sub-manifold of $X$ and we denote it by $X_\bbR$. The moment coordinates $(x_1,x_2)$ yield a coordinate system on the open dense subset of $X_\bbR$ given by $\bx^{-1}(\interior(P))\cap X_\bbR.$ The metric $g_K$ restricts to $X_\bbR$. In the coordinates  $(x_1,x_2)$ it is given by:
$$
g_\bbR=\sum_{i,j=1}^2 u_{ij}dx_i\otimes dx_j.
$$
Alternatively, it will be useful to think of $\bx$ as a Riemannian submersion onto $\interior(P)$ equipped with a metric with the same formula as $g_\bbR$ which we shall still denote by $g_\bbR$. Metric related quantities can be expressed in terms of $u$ and throughout the paper we will use the following formulas:
\begin{itemize}
\item Let $h:X\rightarrow \bbR$ be a smooth, $\bbT^2$-invariant function on $X$. Then $h$ can be thought of as a function on $P.$ In coordinates on the open set  $\bx^{-1}(\interior(P))$, $h$ depends only on $(x_1,x_2)$. The Laplacian of $h$ is given by:
\begin{equation}\label{lapgK}
\Lap h=\sum_{i,j=1}^2\frac{\partial}{\partial x_i}\left( u^{ij}\frac{\partial h}{\partial x_j}\right),
\end{equation}
on $\bx^{-1}(\interior(P)).$
\item Such a smooth $\bbT^2$-invariant function $h:X\rightarrow \bbR$ is pulled back from $P$ and so we also take its Laplacian with respect to the metric $g_\bbR$ on $\interior(P)$ which we denote by $\Lap_\bbR h.$ It is given by
$$
\Lap_\bbR h=\frac{1}{\sqrt{\det(\Hess u)}}\sum_{i,j=1}^2\frac{\partial}{\partial x_i}\left( \sqrt{\det(\Hess u)} u^{ij}\frac{\partial h}{\partial x_j}\right).
$$

\item The scalar curvature of $g_K$ is given by the Abreu formula,
$$
\scal(g_K)=\sum_{i,j=1}^2\frac{\partial^2 u^{ij}}{\partial x_i \partial x_j}.
$$
This can equally be written by making use of $\Lap_\bbR$ as
$$\scal(g_K)= \frac{\Lap_\bbR V}{V},$$
where $V=(\det(\Hess u))^{-\frac{1}{2}}$. See Proposition 3 in \cite{OSD}.
\end{itemize}

\subsection{Extremal toric K\"ahler metrics}
Let $(X^{2n},J,\omega,g_K)$ be a K\"ahler manifold. For a suitable class of functions $\varphi$ on $X,$ $\omega+\sqrt{-1}\partial\bar{\partial} \varphi$ defines another K\"ahler metric compatible with $J.$ In fact all K\"ahler metrics whose cohomology class is $[\omega]$ are of the form $\omega+\sqrt{-1}\partial\bar{\partial} \varphi$ for some $\varphi.$ Calabi defined a functional on the set $\{\varphi\in \mathcal{C}^\infty(X):\omega+\sqrt{-1}\partial\bar{\partial} \varphi \,\text{ is K\"ahler}\}$ by:
$$
\mathcal{C}(\varphi)=\int_X\scal^2(\omega+\sqrt{-1}\partial\bar{\partial} \varphi)(\omega+\sqrt{-1}\partial\bar{\partial} \varphi)^n,
$$
where $\scal(\omega+\sqrt{-1}\partial\bar{\partial} \varphi)$ means the scalar curvature of the Riemannian metric associated with $\omega+\sqrt{-1}\partial\bar{\partial} \varphi.$
A K\"ahler metric in $[\omega]$ is said to be extremal (in the sense of Calabi) if it is a critical point for $\mathcal{C}.$ 

It is well know that a K\"ahler metric $\omega$ is extremal in the sense of Calabi in its cohomology class iff $\nabla^{(1,0)} \scal(\omega)$, the $(1,0)$ part of the gradient of $ \scal(\omega)$, is a holomorphic vector field. A special class of extremal metrics is the class of constant scalar curvature K\"ahler metrics. In the toric setting, it can be shown that $(X^4,\omega,g_K)$ is extremal iff $\scal(g_K)$ is an affine function of the action coordinates $(x_1,x_2).$

We known that not all K\"ahler manifolds admit extremal metrics. For toric manifolds, independently of existence of extremal metrics, it is possible to calculate the only affine function that may be the scalar curvature of an extremal metric in case it exits. This is called the extremal affine function. We explain how to determine it for toric surfaces since it will be relevant for our purposes. Let $X$ be a toric K\"ahler manifold of real dimension $4$ with moment polytope $\Delta.$ As before write 
$$
\Delta=\{(x_1,x_2)\in \bbR^2: l_i(x_1,x_2):=\langle (x_1,x_2),\nu_i\rangle- \lambda_i\geq 0, i=1,\cdots, d\}.
$$
Following \cite{do2}, we are going to define a measure $d\sigma$ on the boundary of $\Delta.$ On the edge {$l_i^{-1}(0)\cap \Delta$, which we denote by $E_i$} and whose interior normal is $\nu_i=(\nu_i^1,\nu_i^2),$ we define $d\sigma$ so that
$$
d\sigma\wedge( \nu_i^1dx_1+\nu_i^2dx_2)=dx_1\wedge dx_2.
$$
The extremal affine function $\alpha_0+\alpha_1x_1+\alpha_2x_2$ is the only affine function on $\Delta$ satisfying
$$
\int_\Delta (\alpha_0+\alpha_1x_1+\alpha_2x_2)h(x_1,x_2)dx_1dx_2=\int_{\partial \Delta}h(x_1,x_2)d\sigma,
$$
for all affine functions $h$ on $\Delta.$ This can be expressed as
$$
\begin{pmatrix}
\int_\Delta 1 &\int_\Delta x_1&\int_\Delta x_2 \\
\int_\Delta x_1 &\int_\Delta (x_1)^2&\int_\Delta x_1x_2 \\
\int_\Delta x_2 &\int_\Delta x_1x_2&\int_\Delta (x_2)^2 \\
\end{pmatrix}
\begin{pmatrix}
\alpha_0\\
\alpha_1\\
\alpha_2
\end{pmatrix}=
\begin{pmatrix}
\int_{\partial \Delta} d\sigma  \\
\int_{\partial \Delta} x_1 d\sigma  \\
\int_{\partial \Delta} x_2 d\sigma   \\
\end{pmatrix}.
$$

We can translate $\Delta$ so that $\alpha_0=0$ in which case we write $\fut=(\alpha_1,\alpha_2)$ and we call $\fut$ the extremal vector.

\begin{rem}
	Let $(M,\omega,J)$ be a compact toric K\"ahler manifold with moment polytope $P$ and $D$ a torus invariant divisor. Let $E$ be the image of $D$ in $P$. Then, in this article we will be interested in using the setup of this section in the case where $X=M \setminus D$ and $\Delta =P\setminus E$. The extremal affine function of $M\setminus D$ is the extremal affine function of $M$ with weight zero along the edge $E.$
\end{rem}

\subsection{Biquard-Gauduchon's construction}

In \cite{bg}, Biquard-Gauduchon construct Hermitian, non-K\"ahler, Ricci-flat ALF toric metrics. Start with $f$ any function of $\bbR$ to $\bbR$ of the form
$$
f(z)=A+\sum_{i=1}^d a_i|z-z_i|,
$$
for some choice of $a_1,\cdots, a_d$ and $z_1<\cdots< z_d,$ and positive $A.$ There is a unique function $U^{BG}$ which is axi-symmetric harmonic in $\bbR^3$ and such that
$$
\lim\limits_{\rho \to 0}\frac{U^{BG}(z,\rho)}{\log (\rho^2)}=f(z).
$$
This can be explicitly written as
\begin{align}\nonumber
U^{BG}(z,\rho)& =2\sum_{i=1}^{d-1} a_i\left(\sqrt{\rho^2+(z-z_i)^2}-(z-z_i)\log\left(\frac{\sqrt{\rho^2+(z-z_i)^2}+z-z_i}{\rho}\right)\right) \\ \label{UBG}
& +A\log(\rho^2).
\end{align}
Here $z$ is the third coordinate in $\bbR^3$ and $\rho$ is the distance to the $z$-axis.
One of the results in \cite{bg} is the following.

\begin{thm}[Biquard-Gauduchon]\label{thmBG2}
Let $f$ be a function given by
$$
f(z)=A+\sum_{i=1}^d a_i|z-z_i|,
$$
for $A>0$ and ${a_1,\cdots,a_d>0}.$ Let $U^{BG}$ be the axi-symmetric harmonic function of $(z,\rho)$ defined on the whole upper half-plane and such that 
$$
\lim\limits_{\rho \to 0}\frac{U^{BG}(z,\rho)}{\log (\rho^2)}=f(z).
$$
{ There is a local K\"ahler toric metric $g^{BG}_K$ with symplectic potential $u^{BG},$ such that:}
$$
\rho U^{BG}_{\rho}=\frac{1}{\scal(g^{BG}_K)}\qquad U^{BG}_z=\det(\triangledown u^{BG}, \fut),
$$
where $\fut$ is the vector whose coefficients are the coefficients of the linear part of $\scal(g^{BG}_K)$. This metric is extremal and Bach-flat. Moreover, $g^{BG}=\frac{g^{BG}_K}{\scal^2(g^{BG}_K)}$ is Ricci-flat and locally toric. 

{There is a toric manifold of the form $M\setminus D,$ for a torus invariant divisor $D,$ such that the metric $g^{BG}_K$ extends smoothly to $M\setminus D$ only when $g^{BG}_K$ coincides with one of metrics in Theorem (\ref{thmBG}).}
\end{thm}

{The above is a consequence of Corollary (3.6), Corollary (5.2) and Lemma (8.1) in \cite{bg}. The metrics arising from the above theorem were all known but the theorem gives a new unified construction for all the metrics appearing in Theorem (\ref{thmBG}).}
\section{An extremal Bach-flat toric K\"ahler metric from $g$}
Let $M$ be a toric K\"ahler manifold with moment polytope 
$$
P=\{(x_1,x_2)\in \bbR^2: l_i(x_1,x_2):=\langle (x_1,x_2),\nu_i\rangle- \lambda_i\geq 0, i=1,\cdots, d+1\}.
$$
This is the image of the moment map $\bx:M\rightarrow \bbR^2$ and we shall denote its edges by
$$E_i=l_i^{-1}(0), \ \ i \in \lbrace 1, \ldots , d+1 \rbrace ,$$  
and when no confusion may arise, we shall denote the last of these edges by $E:=E_{d+1}$. Let $D=\bx^{-1}(E)$ be the toric divisor corresponding this last edge of $P.$ Assume that $X=M\setminus D$ is endowed with a conformally K\"ahler, Ricci-flat, toric  metric $g.$ The goal of this section is to prove the following.
\begin{lemma}
Let $g$ be a conformally K\"ahler, Ricci-flat, toric metric on $X.$ Then, there is a K\"ahler toric metric $g_K$ which is extremal and Bach-flat  to which $g$ is conformal.
\end{lemma}
The proof essentially follows from Theorem (\ref{Thmderd1}) and Theorem (\ref{Thmderd2}).
\begin{proof} The first thing to note is that since $g$ is Ricci-flat it is Einstein and therefore it satisfies $\delta W^+=0.$ This is because
$$
\nabla^p W_{pkij}=\nabla_i\mP_{kj}-\nabla_j\mP_{ki},
$$
with $\mP=\frac{1}{2}\left(\text{Ric}-\frac{1}{6}\scal\right)$ being the Einstein tensor. In the Einstein case this is zero so that from the formula above $\delta W=0$. But 
$$
|\delta W|^2=|\delta W^+|^2+|\delta W^-|^2
$$
and so $\delta W^+=0.$ Our metric $g$ on $X=M\setminus D$ is Ricci-flat and Hermitian and by the Riemannian Golberg--Sachs theorem \cite{AG,l0} it is half-algebraically simple. We can thus apply Theorem (\ref{Thmderd2}) to $g$. Let $\lambda$ be the simple eigenvalue of $W^+(g)$ and $\omega$ the corresponding eigenform. Because $g$ is $\bbT^2$-invariant, so is $W^+$ and therefore $\lambda$ and $\omega$ are $\bbT^2$-invariant as well. The K\"ahler metric $g_K=\lambda^{2/3}g$ is therefore $\bbT^2$-invariant. We have $g=\lambda^{-2/3}g_K$ and therefore, combining the second bullet in Theorem (\ref{Thmderd1}) $\lambda^{1/3}=\scal(g_K)$ and
$$
g=\frac{g_K}{\scal^2(g_K)}.
$$
for a toric K\"ahler metric $g_K.$ This metric must be extremal and Bach-flat from the third and fourth bullets in Theorem (\ref{Thmderd1}).
\end{proof}

\section{Action-angle coordinate setup}
From the previous section there is a K\"ahler toric metric $g_K$ which is extremal and Bach-flat  to which the Ricci-flat metric $g$ is conformal. We can apply the toric K\"ahler framework to the metric $g_K.$ There is an open dense set in $X=M\setminus D$ and action-angle coordinates $(x_1,x_2,\theta_1,\theta_2)$ on that set and a symplectic potential $u$ for $g_K$ such that 
$$
g_K=\sum_{i,j=1}^2 \left( u_{ij}dx_i\otimes dx_j+u^{ij}d\theta_i\otimes d\theta_j \right),
$$
and recall that we have an induced metric on the polytope which we denoted by $g_\bbR$. Because $g_K$ is extremal, $\scal(g_K)$ is an affine function of $(x_1,x_2).$ Without loss of generality we assume it vanishes at zero so that
$$
\scal(g_K)=\alpha_1x_1+\alpha_2x_2
$$
with $\fut=(\alpha_1,\alpha_2)$ being the extremal vector field. 
Let 
$$
\rho(x_1,x_2)=\frac{(\det(\Hess{u}))^{-1/2}}{\scal^2(g_K)}.
$$
Then, both $\rho$ and $\scal(g_K)$ are $\bbT^2$-invariant functions, which are therefore pulled back from $P$, and we can consider taking their Laplacian with respect to the metric $g_\bbR$ on $\interior(P)$. This leads to the following lemma.
\begin{lemma}\label{lemma_harmonic}
The function $\rho$ is harmonic for $g_\bbR$.
\end{lemma}
\begin{proof}
This proof is a really a calculation.
$$
\Lap_\bbR \rho=\frac{1}{\sqrt{\det(\Hess u)}} \sum_{i,j=1}^2\frac{\partial}{\partial x_i}\left( \sqrt{\det(\Hess u)}u^{ij}\frac{\partial \rho}{\partial x_j}\right),
$$
so that $\Lap_\bbR \rho=0$ iff
$$
\sum_{i,j=1}^2\frac{\partial}{\partial x_i}\left( \sqrt{\det(\Hess u)}u^{ij}\frac{\partial \rho}{\partial x_j}\right)=0.
$$
Now $-\sum_{i,j=1}^2\frac{\partial}{\partial x_i}\left( \sqrt{\det(\Hess u)}u^{ij}\frac{\partial \rho}{\partial x_j}\right)$ is given by
\begin{IEEEeqnarray*}{rCl}
& &\sum_{i,j=1}^2\frac{\partial}{\partial x_i}\Biggl( \sqrt{\det(\Hess u)}u^{ij}\left( \frac{(\det(\Hess{u}))^{-3/2}}{2\scal^2 (g_K)}\frac{\partial \det(\Hess u)}{\partial x_j}+ \right.\\ \nonumber
&&\left. \frac{2(\det(\Hess{u}))^{-1/2}}{\scal^3 (g_K)}\frac{\partial \scal (g_K)}{\partial x_j}\right)\Biggr)\\ 
& =&\sum_{i,j=1}^2\frac{\partial}{\partial x_i}\left( u^{ij}\left( \frac{1}{2\scal^2 (g_K)}\frac{\partial \log(\det(\Hess u))}{\partial x_j} +\frac{2}{\scal^3 (g_K)}\frac{\partial \scal (g_K)}{\partial x_j}\right)\right)\\  \nonumber
&=&\sum_{i,j=1}^2\frac{\partial}{\partial x_i}\left( u^{ij}\left( \frac{1}{2\scal^2 (g_K)}\sum_{a,b=1}^2u^{ab}u_{abj} +\frac{2}{\scal^3 (g_K)}\frac{\partial \scal (g_K)}{\partial x_j}\right)\right),\\ \nonumber
\end{IEEEeqnarray*}
where
$$
u_{ijk}=\frac{\partial^3 u}{\partial x_i\partial x_j\partial x_k} , \qquad \forall i,j,k=1,2.
$$
This is because for any smooth family of invertible matrices $A(t)$ we have
$$
\frac{d\log\det A(t)}{dt}=\tr\left(A^{-1}(t)\frac{d A(t)}{dt}\right).
$$
We also have 
$$
\frac{d A(t)^{-1}}{dt}=-A\frac{d A(t)}{dt}A^{-1},
$$
hence 
$$
\sum_{a,b,j=1}^2u^{ij}u^{ab}u_{abj}=-\sum_{k=1}^2u^{ik}_k
$$
where 
$$
u^{ij}_k=\frac{\partial u^{ij}}{\partial x_l}, \qquad \forall i,j,k=1,2.
$$
So $\Lap \rho=0$ iff
$$
\sum_{i,j=1}^2\frac{\partial}{\partial x_i}\left(-\frac{\sum_{k=1}^2 u^{ik}_k}{2\scal^2(g_K)} +\frac{2u^{ij}}{\scal^3(g_K)}\frac{\partial \scal (g_K)}{\partial x_j}\right)=0.
$$
This is equivalent to
\begin{IEEEeqnarray*}{rCl}
&&-\frac{\sum_{i,k=1}^2 u^{ik}_{ik}}{2\scal^2(g_K)} +\frac{\sum_{i,k=1}^2 u^{ik}_k}{\scal^3(g_K)}\frac{\partial \scal (g_K)}{\partial x_i}+
\sum_{j=1}^2 \frac{\sum_{i=1}^2 2u^{ij}_i}{\scal^3(g_K)}\frac{\partial \scal (g_K)}{\partial x_j}\\
&&+\sum_{i,j=1}^2 \frac{2u^{ij}}{\scal^3 (g_K)}\frac{\partial^2 \scal (g_K)}{\partial x_i\partial x_j}-6\sum_{i,j=1}^2 \frac{2u^{ij}}{\scal^4 (g_K)}\frac{\partial \scal (g_K)}{\partial x_i}\frac{\partial \scal (g_K)}{\partial x_j}=0. \\
\end{IEEEeqnarray*}

Because our metrics are extremal 
$$
\frac{\partial^2 \scal (g_K)}{\partial x_i\partial x_j}=0, \qquad \forall i,j=1,2.
$$
From Abreu's equation 
$$
\sum_{i,k=1}^2 u^{ik}_{ik}=-\scal (g_K).
$$
The above simplifies to yield 
\begin{equation}\label{lap=0}
\frac{1}{2\scal(g_K)} +\frac{3\sum_{i,k=1}^2 u^{ik}_k}{\scal^3 (g_K)}\frac{\partial \scal (g_K)}{\partial x_i}-6\sum_{i,j=1}^2 \frac{2u^{ij}}{\scal^4 (g_K)}\frac{\partial \scal (g_K)}{\partial x_i}\frac{\partial \scal (g_K)}{\partial x_j}=0. \\
\end{equation}
Next we make use of the Bach-flatness condition. It is known that Bach-flatness in our setting is equivalent to
$$
\scal \left(\frac{g_K}{\scal^2 (g_K)}\right)
$$
being constant. Because $\frac{g_K}{\scal^2 (g_K)}$ is Ricci-flat, the constant is zero therefore Bach-flatness is equivalent to
$$
\scal \left(\frac{g_K}{\scal^2 (g_K)}\right)=0
$$
{Now, using the notation $\Lap_{g_K}$ to denote the Laplacian on $X$ with respect to $g_K$ to avoid confusion, we have}
$$
\scal \left(\frac{g_K}{\scal^2 (g_K)}\right)=(\scal (g_K))^3\left(1-6\Lap_{g_K}\left(\frac{1}{\scal (g_K)}\right)\right),
$$
and we get
$$
\Lap_{g_K}\left(\frac{1}{\scal (g_K)}\right)=\frac{1}{6}.
$$
Let us make this explicit by using Equation (\ref{lapgK}).
\begin{IEEEeqnarray*}{rCl}
\frac{1}{6}&=&\sum_{i,j=1}^2\frac{\partial}{\partial x_i}\left(u^{ij}\frac{\partial }{\partial x_i}\left(\frac{1}{\scal (g_K)}\right)\right)\\ \nonumber
&=&-\sum_{i,j=1}^2\frac{\partial}{\partial x_i}\left(\frac{u^{ij}}{\scal^2 (g_K)}\frac{\partial \scal (g_K) }{\partial x_i}\right)\\ \nonumber
&=&-\sum_{j=1}^2\frac{\sum_{i=1}^2u^{ij}_i}{\scal^2 (g_K)}\frac{\partial \scal (g_K) }{\partial x_i} \\
& & -\sum_{i,j=1}^2\frac{u^{ij}}{\scal^2 (g_K)}\frac{\partial^2 \scal (g_K) }{\partial x_i\partial x_j}+\sum_{i,j=1}^2\frac{2u^{ij}}{\scal^3 (g_K)}\frac{\partial \scal (g_K) }{\partial x_i}\frac{\partial \scal (g_K) }{\partial x_j}\\ \nonumber
&=&-\sum_{j=1}^2\frac{\sum_{i=1}^2u^{ij}_i}{\scal^2 (g_K)}\frac{\partial \scal (g_K) }{\partial x_i}+\sum_{i,j=1}^2\frac{2u^{ij}}{\scal^3 (g_K)}\frac{\partial \scal (g_K) }{\partial x_i}\frac{\partial \scal (g_K) }{\partial x_j}\\ \nonumber
\end{IEEEeqnarray*}
for extremal metrics. Replacing in Equation (\ref{lap=0}) we get
$$
\begin{aligned}
&\frac{1}{2\scal(g_K)} +\frac{3\sum_{i,k=1}^2 u^{ik}_k}{\scal^3 (g_K)}\frac{\partial \scal (g_K)}{\partial x_i}-6\sum_{i,j=1}^2 \frac{2u^{ij}}{\scal^4 (g_K)}\frac{\partial \scal (g_K)}{\partial x_i}\frac{\partial \scal (g_K)}{\partial x_j}\\
&=\frac{1}{2\scal(g_K)} -\frac{1}{2\scal(g_K)} \\
&=0,
\end{aligned}
$$
and we are done.
\end{proof}
Because $\rho$ is harmonic for $g_\bbR$ it admits a harmonic conjugate which we denote by $z$ so that
$$
\sum_{i,j=1}^2u_{ij}dx_i\otimes dx_j =e^{2v}\left(d\rho^2+dz^2\right).
$$
Next, we state an important but rather technical lemma. The proof goes exactly as the proof of Proposition 5.1 in \cite{s} so we do not go over it here.
\begin{lemma} \label{lemma_bij} 
Let $M$ be a toric K\"ahler surface with moment polytope $P$ and $D$ be a torus invariant divisor in $X$ whose image via the moment map is $E.$ Let $g$ be a conformally K\"ahler toric metric conformal to an extremal Bach-flat toric K\"ahler metric $g_K$ with symplectic potential $u.$ Consider also the map $\bz=z+\sqrt{-1}\rho: P\setminus E\rightarrow \bbH$ where 
$$
\rho=\frac{(\det \Hess u)^{-1/2}}{\scal(g_K)} 
$$ 
and $z$ is the harmonic conjugate of $\rho$ for $g_K$ restricted to the real manifold. Then $\bz$ is a bijection. 
\end{lemma}
\section{The function $U$}
As we have seen in Theorem (\ref{thmBG2}),  Biquard and Gauduchon make use of an axi-symmetric harmonic function on $\bbR^3$ in their construction for conformally K\"ahler, Ricci-flat metrics. This function, which we have denoted by $U^{BG},$ is given explicitly in their work by Equation (\ref{UBG}). In this section, our aim is to see that conversely, an axi-symmetric harmonic function $U$ can be associated to any conformally K\"ahler, Ricci-flat, toric metric on $M\setminus D.$  What we do here is very similar to what Donaldson does in \cite{donjoyce} to show that every toric scalar-flat K\"ahler metric arises from Joyce's ansatz. The metric $g_K$ and its symplectic potential have specified behaviour at $\partial P\setminus E$ and this will imply a specific boundary behaviour for $U$ on $\partial \bbH$.

Consider the map $\bz=z+\sqrt{-1}\rho$ in Lemma (\ref{lemma_bij}). Even though the functions ${\scal(g_K)}$ and $\det(\triangledown u, \fut)$ are functions on $P\setminus E,$ we interpret them as functions on $\bbH$ via Lemma (\ref{lemma_bij}) i.e. 
$$
\begin{aligned}
\scal(g_K)(z+\sqrt{-1}\rho)&=\scal(g_K)\circ \bz^{-1}(z+\sqrt{-1}\rho),\\ \nonumber
 \det(\triangledown u, \fut) (z+\sqrt{-1}\rho)&=\det(\triangledown u, \fut)\circ \bz^{-1} (z+\sqrt{-1}\rho)\\ \nonumber
 \end{aligned}
$$
We shall use Lemma (\ref{lemma_bij}) to show the following proposition. 
\begin{prop}\label{U}
Let $M$ be a toric K\"ahler surface with moment polytope $P$ and $D$ be a torus invariant divisor in $X$ whose image via the moment map is an edge $E.$ Let $g$ be a Ricci-flat conformally K\"ahler, toric metric conformal to an extremal Bach-flat toric K\"ahler metric $g_K$ with symplectic potential $u.$ Let $P$ be normalised so that $\scal(g_K)=\alpha_1x_1+\alpha_2x_2$ for constants $\alpha_1,\alpha_2$. Set $\fut=(\alpha_1,\alpha_2).$  Then, there is a function $U:\bbH \rightarrow \bbR$ such that 
$$
\rho U_\rho=\frac{1}{\scal(g_K)(z+\sqrt{-1}\rho)},\qquad U_z=\det(\triangledown u, \fut) (z+\sqrt{-1}\rho).
$$
The function $U$ is axi-symmetric harmonic i.e. it satisfies
\begin{equation}\label{Uharmonic}
U_{\rho\rho}+U_{zz}+\frac{U_\rho}{\rho}=0.
\end{equation}
The function $\frac{U}{\log(\rho^2)}$ extends to $\partial \bbH$ and it is affine with slope $\det(\fut,\nu_i)$ when restricted to each segment of $\partial \bbH$ corresponding to the image via $\bz$ of the $i$th facet of $P\setminus E,$ for $i=1,\cdots,d.$
\end{prop}
\begin{proof}
The first point is to show that the two equations are compatible which boils down to showing that 
$$
\frac{\partial}{\partial z }\left( \frac{1}{\rho \scal(g_K)}\right)=\frac{\partial \det(\triangledown u, \fut)}{\partial \rho },
$$
where both functions are thought of as functions as $z+\sqrt{-1}\rho$ but we omit explicitly writing the argument. We start by relating $z+\sqrt{-1}\rho$ with toric complex coordinates. Recall that complex toric coordinates on $X$ are given by $\xi_i+\sqrt{-1}\theta_i$ for $i=1,2$ where $\bf{\xi}=(\xi_1,\xi_2)$ is
$$
\bf{\xi}(\bf{x})=\triangledown u(\bf{x}).
$$
We are going to follow an argument of Donaldson's from \cite{donjoyce}. Consider the metric on the real manifold $g_\bbR.$ We shall express it using complex  coordinates.
\begin{IEEEeqnarray*}{rCl}
\sum_{i,j=1}^2 u_{ij}dx_i\otimes dx_j&=&\sum_{i=1}^2 d\xi_i\otimes dx_i\\ \nonumber
&=&(\xi_{1,z}dz+\xi_{1,\rho}d\rho)\otimes(x_{1,z}dz+x_{1,\rho}d\rho)\\
& & \quad +(\xi_{2,z}dz+\xi_{2,\rho}d\rho)\otimes(x_{2,z}dz+x_{2,\rho}d\rho)\\ \nonumber
&=&(\xi_{1,z}x_{1,z}+\xi_{2,z}x_{2,z})dz^2+ (\xi_{1,\rho}x_{1,\rho}+\xi_{2,\rho}x_{2,\rho})d\rho^2\\ \nonumber
& &\quad  +(\xi_{1,z}x_{1,\rho}+\xi_{2,z}x_{2,\rho}+\xi_{1,\rho}x_{1,z}+\xi_{2,\rho}x_{2,z})dz\otimes d\rho. \\ \nonumber
\end{IEEEeqnarray*}
We set
$$
\mA=\begin{pmatrix}
x_{1,z}& x_{1,\rho}\\
x_{2,z}&x_{2,\rho}
\end{pmatrix},
 \qquad 
 \mB=\begin{pmatrix}
\xi_{1,z}& \xi_{1,\rho}\\
\xi_{2,z}&\xi_{2,\rho}
\end{pmatrix}.
$$
We start by noticing that $\mA^{T}\mB$ is a symmetric matrix. This is because
$$
\mB=\frac{\partial \bf{\xi}}{\partial(z,\rho)}=\frac{\partial \bf{\xi}}{\partial \bf{x}}\frac{\partial \bf{x}}{\partial(z,\rho)}\\ \nonumber=\Hess (u) \mA
$$
since $\bf{\xi}(\bf{x})=\triangledown u(\bf{x})$. Therefore $$\mA^{T}\mB=\mA^{T}\Hess (u) \mA,$$ which is symmetric. On the other hand we have
$$
\sum_{i,j=1}^2 u_{ij}dx_i\otimes dx_j =e^{2v}(dz^2+d\rho^2),
$$ 
which implies $$\mA^T\mB=e^{2v}I.$$ Now by taking determinants
$$
e^{4v}=\det(\mA)\det(\mB)=(\det(\mA))^2\det(\Hess(u))
$$
so that $e^{2v}=\det(\mA)(\det(\Hess(u))^{1/2}$ and this implies that $$\mB=\det(\mA)(\det(\Hess(u))^{1/2}{(\mA^{-1})^T},$$ i.e.
$$
\begin{pmatrix}
\xi_{2,\rho}& \xi_{1,\rho}\\
\xi_{2,z}&\xi_{1,z}
\end{pmatrix}=\det(\Hess(u))^{1/2}\begin{pmatrix}
x_{1,z}& -x_{2,z}\\
-x_{1,\rho}&x_{2,\rho}
\end{pmatrix}.
$$
We rewrite this as
\begin{equation}\label{xi&x}
\begin{cases}
\xi_{1,z}=\frac{1}{\rho}\frac{x_{2,\rho}}{\scal^2(g_K)}\\
\xi_{1,\rho}=-\frac{1}{\rho}\frac{x_{2,z}}{\scal^2(g_K)}\\
\end{cases}
\begin{cases}
\xi_{2,z}=-\frac{1}{\rho}\frac{x_{1,\rho}}{\scal^2(g_K)}\\
\xi_{2,\rho}=\frac{1}{\rho}\frac{x_{1,z}}{\scal^2(g_K)}\\
\end{cases},
\end{equation}
i.e. 
$$
\begin{cases}
\xi_{1,z}=\frac{1}{\rho}\frac{x_{2,\rho}}{(\alpha_1x_1+\alpha_2x_2)^2}\\
\xi_{1,\rho}=-\frac{1}{\rho}\frac{x_{2,z}}{(\alpha_1x_1+\alpha_2x_2)^2}\\
\end{cases}
\begin{cases}
\xi_{2,z}=-\frac{1}{\rho}\frac{x_{1,\rho}}{(\alpha_1x_1+\alpha_2x_2)^2}\\
\xi_{2,\rho}=\frac{1}{\rho}\frac{x_{1,z}}{(\alpha_1x_1+\alpha_2x_2)^2}\\
\end{cases}.
$$
Using this we have
\begin{IEEEeqnarray*}{rCl}
-\frac{\partial \det( \triangledown u,\fut)}{\partial \rho }&=&\frac{\partial \left(-\alpha_2\xi_{1}+\alpha_1\xi_{2}\right)}{\partial \rho }\\ \nonumber
&=&-\alpha_2\xi_{1,\rho}+\alpha_1\xi_{2,\rho}\\ \nonumber
&=&\frac{\alpha_2x_{2,z}}{\rho(\alpha_1x_1+\alpha_2x_2)^2}+\frac{\alpha_1x_{1,z}}{\rho(\alpha_1x_1+\alpha_2x_2)^2}\\ \nonumber
&=&\frac{\frac{\partial}{\partial z}\left(\alpha_1x_1+\alpha_2x_{2}\right)}{\rho(\alpha_1x_1+\alpha_2x_2)^2}\\ \nonumber
&=&-\frac{\partial}{\partial z}\left(\frac{1}{\rho(\alpha_1x_1+\alpha_2x_2)}\right)\\ \nonumber
&=&-\frac{\partial}{\partial z}\left(\frac{1}{\rho \scal(g_K)}\right)\\ \nonumber
\end{IEEEeqnarray*}
which shows that $U$ exists as $\bbH$ is simply connected. Next we show $U$ satisfies Equation (\ref{Uharmonic}).
\begin{IEEEeqnarray*}{rCl}
U_{\rho\rho}+U_{zz}+\frac{U_\rho}{\rho}&=&\frac{\partial}{\partial \rho}\left( \frac{1}{\rho\scal(g_K)}\right)+\frac{\partial \det( \triangledown u,\fut)}{\partial z}+ \frac{1}{\rho^2\scal(g_K)}\\ \nonumber
&=& -\frac{1}{\rho^2\scal(g_K)}- \frac{\alpha_1x_{1,\rho}+\alpha_2x_{2,\rho}}{\rho\scal^2(g_K)} {+\alpha_2\xi_{1,z}-\alpha_1\xi_{2,z}}+\frac{1}{\rho^2\scal(g_K)}\\ \nonumber
&=& - \frac{\alpha_1x_{1,\rho}+\alpha_2x_{2,\rho}}{\rho\scal^2(g_K)} {+\alpha_2\xi_{1,z}-\alpha_1\xi_{2,z}}\\ \nonumber
&=&0
\end{IEEEeqnarray*}
where we made use of Equation (\ref{xi&x}) in the last step.
Next we will focus on the boundary behaviour of $U$. We have 
$$
\rho=\frac{(\det \Hess(u))^{-1/2}}{\scal^2 (g_K)}.
$$
As we have seen 
$$
\det \Hess(u)(\bx)=\frac{\delta(\bx)}{\prod_{i=1}^d l_i(\bx)}, \qquad \bx \in P,
$$
for a positive smooth function $\delta$ on $P\setminus E,$ so that $\rho$ vanishes at $\partial P\setminus E.$ Also near the $i$-the edge of $P\setminus E,$ $2\log(\rho)-\log(l_i)$ is smooth. What is more,
$$
\rho U_{\rho}=\frac{1}{\scal(g_K)}=S_1(z)+ \rho \tilde{S}_2(z,\rho),
$$
so that
$$
U(z,\rho)=S_0(z)+S_1(z)\log(\rho)+ \rho S_2(z,\rho),
$$
where $S_0, S_1, S_2 , \tilde{S}_2$ are functions of $z$, smooth in the interior of segments corresponding to images of facets or $P\setminus E$ via the map $\bf{z},$ and $S_2$ is smooth in a neighbourhood of those segments in $\bbH.$  In particular
$$
\lim_{\rho\rightarrow 0}\frac{U(z,\rho)}{\log(\rho)}=S_1(z).
$$
Now
$$
U_z= \det( \triangledown u,\fut)=S'_0(z)+S'_1(z)\log(\rho)+\rho \frac{\partial S_2}{\partial z}(z,\rho).
$$
Because $u-\sum_{i=1}^dl_i\log(l_i)$ is smooth near the ith facet of $P\setminus E$ for $i=1,\cdots, d$,
$$
\det(\triangledown u-\log(l_i)\nu_i,\fut)
$$
is smooth as well. Near the ith facet $2\log(\rho)-\log(l_i)$ is smooth so 
$$
\det(\triangledown u-2\log(\rho)\nu_i,\fut)
$$
is smooth and this implies that
$$
S'_0(z)+S'_1(z)\log(\rho)+\rho \frac{\partial S_2}{\partial z}(z,\rho)-\log(\rho^2)\det (\nu_i,\fut)
$$
is smooth therefore
$$
S'_1(z)=-2\det (\nu_i,\fut)
$$ 
is constant on the interior of the image of the $i$th edge and $S_1(z)$ is affine on the interior of the image of that edge as claimed.
\end{proof}

{Next we will show that the vanishing locus of $\scal(g_K)$ is highly restricted. The lemma below although simple plays a very important role in our results. In particular it explains the condition on the extremal vector field appearing in Theorem (\ref{thm_main2}).
\begin{lemma}
In the setting of Proposition (\ref{U}) let $\fut$ be the extremal vector field associated with $P\setminus E.$ There are three possibilities for the vanishing locus of $\scal(g_K).$
\begin{itemize}
\item If $\det(\fut,\nu_1),\det(\fut,\nu_{d})\ne 0$ i.e. if $\fut$ is not perpendicular to either one of the non-compact edges of $P\setminus E$ then $\scal(g_K)$ vanishes along $D.$
\item If exactly one of the two numbers $\det(\fut,\nu_1), \det(\fut,\nu_{d})$ is zero, then $\scal(g_K)$ vanishes at a single point in $D$ corresponding to either $E\cap E_d$ or $E\cap E_1$ via the moment map.
\item If $\det(\fut,\nu_1)=0=\det(\fut,\nu_{d})$ then $\scal(g_K)$  is nowhere vanishing over $X.$
\end{itemize}
\end{lemma}
\begin{proof}
Let us use the same notation as in Proposition (\ref{U}). As we have just seen, the function
$$
f(z):=\lim_{\rho\rightarrow 0}\frac{U(z,\rho)}{\log(\rho^2)}
$$
is locally affine with slope $\det(\fut,\nu_i)$ over $]z_{i-1},z_i[$ where by convention $z_0=-\infty,$ $z_d=+\infty$ and $i=1,\cdots, d.$ On the other hand it follows from the proof above that 
$$ f(z)=\frac{1}{2\scal(g_K)(z,0)}.$$
Note also that $\scal(g_K)$ can be thought of as affine function on $P \subset \bbR^2$ whose vanishing locus is a line normal to $\fut.$ Furthermore, the limits $z \to \pm \infty$ correspond to limits along the non-compact ends of the edges $E_1$ and $E_d$
\begin{itemize}
\item If $\det(\fut,\nu_1),\det(\fut,{\nu_{d}})\ne 0$ we see that 
$$
\lim_{z\rightarrow \pm\infty}\frac{1}{|\scal(g_K)|(z,0)}= \lim_{z\rightarrow \pm\infty}|f(z)| =\infty
$$
so that 
$$
\lim_{z\rightarrow \pm\infty}{\scal(g_K)(z,0)}=0.
$$
This shows that $\scal(g_K)$ vanishes at the vertices $E\cap E_1$ and {$E\cap E_{d}$} so that it must vanish along $E$ because it is affine.
\item If exactly one of the two numbers $\det(\fut,\nu_1), \det(\fut,\nu_{d})$ is zero then the same reasoning shows that $\scal(g_K)$ vanishes at exactly one the vertices $E\cap E_1,$   $E\cap E_{d}.$ The extremal vector field is proportional to either $\nu_{d-1}$ or $\nu_{1}$ so the line $\scal(g_K)=0$ is parallel to either $E_1$ if  $\det(\fut,\nu_1)=0$ or $E_{d}$ if $\det(\fut,\nu_d)=0$ and intercepts $\overline{P}$ at a single vertex.
\item If $\det(\fut,\nu_1)=\det(\fut,\nu_{d})=0$ then the edges $E_1$ and $E_{d}$ must be parallel and both parallel to the line $\scal(g_K)=0$ in $\bbR^2.$ If this line were to intersect $\overline{P}$ then it would contain points in $P\setminus E$ and so $\scal(g_K)$ would vanish at points in $X$ and this would imply that $g=\frac{g_k}{\scal^2(g_K)}$ had singularities in $X$ which it does not.  
\end{itemize}
\end{proof}}
In comparing the above proposition with \cite{bg}, it is natural to study the boundary behaviour function of $U$ given by
$$
\lim_{\rho\rightarrow 0}\frac{U(z,\rho)}{\log(\rho^2)}.
$$
It follows from Theorem (\ref{thmBG2}) that such a function will yield a $U^{BG}$ if it is of the form
$$
A+\sum_{i=1}^d a_i|z-z_i|,
$$
for $A>0,$ {$a_1,\cdots, a_d>0.$}  Among those functions which are affine in segments and continuous, these are special.

Let $X$ be a toric K\"ahler surface with moment polytope $P$ so that
$$
P=\{(x_1,x_2)\in \bbR^2: l_i(x_1,x_2):=\langle (x_1,x_2),\nu_i\rangle- \lambda_i\geq 0, i=1,\cdots, d+1\},
$$
where $\nu_i$ are the primitive interior normals to the facets of $P.$ Let $D$ be a torus invariant divisor in $X$ whose image via the moment map is $E$ with normal $\nu_{d+1},$ the edge at infinity. Recall that the extremal vector field $\fut$ of $P\setminus E$ is of the form $\alpha_1\frac{\partial}{\partial x_1}+\alpha_2\frac{\partial}{\partial x_2}$ for constants $\alpha_1,\alpha_2.$ 

\begin{lemma}\label{nu1+nud}
Consider a partition of $\partial \bbH$ into $d$ segments $]z_i,z_{i+1}[,$ $i=0,\cdots,d-1$ where we set $z_0=-\infty$ and $z_{d}=+\infty.$ Consider a function $f$ on $\partial \bbH$ which is continuous over $\bbR$ and affine in the interior of the above segments in $\partial \bbH$ and whose slope in $]z_i,z_{i+1}[$ is $\det(\fut, \nu_{i+1}).$ There are constants $a_1,\cdots, a_d,$ {$A$ and $B$ such that 
$$
f(z)=\sum_{i=1}^{d-1}a_i|z-z_i|+A+Bz.
$$ 
Also, $B=0$ iff $\det(\nu_1+\nu_d,\fut)=0.$}
\end{lemma}
\begin{proof}
We start by justifying the last statement assuming the first. Notice that over $]-\infty,z_1[,$ $\sum_{i=1}^{d-1}a_i|z-z_i|$ has slope $-\sum_{i=1}^{d-1}a_i$ whereas it has slope $\sum_{i=1}^{d-1}a_i$ over $]z_{d-1},\infty[.$ Therefore, $f$ will be of the form $A+\sum_{i=1}^{d-1}a_i|z-z_i|$ if and only if its slope at $-\infty$ is the opposite of its slope at $+\infty$, i.e.
$$
\det (\nu_1,\fut)=-\det( \nu_d,\fut).
$$
In any case we have
{
$$
\sum_{i=1}^{d-1}a_i|z-z_i|+A+Bz=
\begin{cases}
\left(-\sum_{i=1}^{d-1}a_i +B\right)z+A_1, \, \text{on} \, ]z_0,z_{1}[,\\
\left(-\sum_{i=1}^{d-1}a_i+2a_1+B\right) z+A_2, \, \text{on} \, ]z_1,z_{2}[,\\
\cdots \\
\left(-\sum_{i=1}^{d-1}a_i+2a_1+\cdots +2a_j+B\right)+A_j,\,  \text{on} \, ]z_j,z_{j+1}[,
\end{cases}
$$
where $A_i$ are constants for $i=1,\cdots d-1,$ and 
$$
f(z)=
\begin{cases}
\det(\fut,\nu_1)z+C_1, \, \text{on} \, ]z_0,z_{1}[,\\
\det(\fut,\nu_2)z+C_2, \, \text{on} \, ]z_1,z_{2}[,\\
\cdots \\
\det(\fut,\nu_{j+1})z+C_j,\,  \text{on} \, ]z_j,z_{j+1}[. 
\end{cases}
$$
where $C_i$ are constants for $i=1,\cdots d-1.$ Set
$$
2a_j={\det(\fut,\nu_{j+1}-\nu_{j})}, \qquad j=1,\cdots, d-1,
$$
and 
$$
2B=\det(\fut,\nu_1+\nu_d).
$$
This choice for the $a_i$'s and $B$ will yield $f(z)=A+Bz+\sum_{i=1}^{d-1} a_i|z-z_i|$ for the right choice of $A.$}
\end{proof}
{Our next lemma will be relevant for the proof of our main theorem.
	
\begin{lemma}\label{lemma:scal>0}
In the setting of Proposition (\ref{U}), $\scal(g_K)>0$ over $X.$
\end{lemma}}
\begin{proof}
{As we have said, a zero of  $\scal(g_K)$ over $X$ is a singularity for $g=\frac{g_K}{\scal^2(g_K)}$ and therefore  $\scal(g_K)$ cannot vanish over $X$ so that its sign remains constant. We shall argue by contradiction and assume that $\scal(g_K)<0.$ Consider $\scal(g_K)$ as an affine function in $\bbR^2$ so that $\scal(g_K)<0$ on $P\setminus E$ means that $\scal(g_K)$ has a negative minimum on ${P}.$ 

We want to start by showing that this minimum must be achieved over $P\setminus E$ so that $\scal(g_K)$ has a negative minimum on $X$. 
\begin{itemize}
\item If $\det(\fut,\nu_1),\det(\fut,\nu_{d})\ne 0$ then $\scal(g_K)$ vanishes along $E$ so its minimum on ${P}$ cannot be achieved on $E$ and must be achieved on $P\setminus E.$
\item If exactly one of the two numbers $\det(\fut,\nu_1), \det(\fut,\nu_{d})$ is zero, say without loss of generality  $\det(\fut,\nu_1),$ then $\scal(g_K)$ vanishes along a line in $\bbR^2$ parallel to $E_1$ through the vertex $E\cap E_{d}.$ As $\scal(g_K)$ is affine, its minimum in $P$ is attained along a line parallel to this that intersects $P.$ All such lines intersect also $P\setminus E$ except perhaps the one through $E\cap E_{d}.$ But on this line $\scal(g_K)$ is zero and therefore not minimal.
\item If $\det(\fut,\nu_1)=\det(\fut,\nu_{d})=0$ then $\scal(g_K)$ vanishes along a line in $\bbR^2$ parallel to $E_1$ and $E_{d}$ and also attains its minimum along one such line. But such a line cannot intercept ${P}$ without intersecting $P\setminus E.$ Therefore the minimum of $\scal(g_K)$ is attained over $X.$
\end{itemize}}
The rest of the argument is due to LeBrun (see \cite{l1}). The fact that the curvature of the metric $\frac{g_K}{\scal^2(g_K)}$ is zero can be expressed as
$$
0=-6\scal (g_K)\Lap\scal (g_K)-12|\grad \scal (g_K)|^2+(\scal (g_K))^3.
$$
At a minimum of $\scal(g_K)$ over $X,$ $\grad \scal(g_K)=0$ and $\Lap\scal(g_K)<0$. But
$$
6\Lap\scal(g_K)=(\scal(g_K))^2>0,
$$
and we get a contradiction.
Hence $\scal (g_K)$ is always positive on $X.$ 
\end{proof}

	The final result of this section is a consequence of the proofs of Lemmas (\ref{nu1+nud}) and (\ref{lemma:scal>0}).
	
\begin{cor}\label{cor:B=0}
	If $\det(\fut,\nu_1),\det(\fut,\nu_{d})\ne 0$, then $2B=\det(\fut,\nu_1+\nu_d)=0$.
\end{cor}
\begin{proof}
	We saw in the proof of Lemma (\ref{lemma:scal>0}) that under these hypothesis the scalar curvature vanishes along the edge $E=E_{d+1}$. Hence, $\eta$ must be proportional to $\nu_{d+1}$. On the other hand, the smoothness of $M$ implies that
	$$\det(\nu_1,\nu_{d+1})=1=\det(\nu_{d+1},\nu_d),$$
	and so
	$$\det(\nu_1+\nu_d,\nu_{d+1})=0,$$
	as claimed.
\end{proof}

\section{Proof of the main theorems}

We are now in a position to prove Theorems (\ref{thm_main}) and (\ref{thm_main2}).

\subsection{Proof of Theorem (\ref{thm_main})}

As before, consider $M$ a compact toric K\"ahler surface with moment polytope $P.$  Let $D$ be a torus invariant divisor in $M$ whose moment map image is a facet $E$ in $P$ which we may assume is the $d+1$ facet.  Assume that $X=M\setminus D$ is endowed with a conformally K\"ahler, Ricci-flat, toric metric $g$ determining a function $U$ as in Proposition (\ref{U}) and a function
$$
f(z)= \lim_{\rho\rightarrow 0}\frac{U(z,\rho)}{\log(\rho^2)}.
$$
From Lemma (\ref{nu1+nud}) we know that there are numbers $a_1,\cdots,a_d,$ constants $A,B$ and $z_1,\cdots z_{d-1}$ points on $\partial \bbH$ such that
{
$$
f(z)=A+Bz+\sum_{i=1}^{d-1}a_i|z-z_i|.
$$
}
{Set $U^{ref}$ to be 
\begin{equation}\label{explicitUBG}
\begin{aligned}
U^{ref}=&2\sum_{i=1}^{d-1} a_i\left(\sqrt{\rho^2+(z-z_i)^2}-(z-z_i)\log\frac{\sqrt{\rho^2+(z-z_i)^2}+z-z_i}{\rho}\right)\\
&+(A+Bz)\log(\rho^2),
\end{aligned}
\end{equation}
{and notice that this too, satisfies
$$
\lim_{\rho\rightarrow 0}\frac{U^{ref}(z,\rho)}{\log(\rho^2)} = f(z).
$$
}
By Theorem (\ref{thmBG2}) this will yield a conformally K\"ahler, Ricci-flat, toric metric if $A>0$ and $a_1,\cdots, a_d>0$ and $B=0$ but we do {not} assume these conditions for now.}

\begin{lemma}
	In the setup described above, we have
	$$U=U^{ref}.$$
\end{lemma}

\begin{proof}
Regardless of the value of the constants, the function $U^{ref}$ is axi-symmetric harmonic i.e. it satisfies Equation (\ref{Uharmonic}) and 
$$
\lim_{\rho\rightarrow 0}\frac{U^{ref}(z,\rho)}{\log(\rho^2)}=f(z).
$$
The main idea of the proof is very similar to the main idea in \cite{s} and goes back to work of Dominic Wright (see \cite{w}). Consider the function
$$
\gamma=\frac{U_\rho^{ref}-U_\rho}{\rho}.
$$
\begin{itemize}
\item We start by proving that $\gamma$ is axi-symmetric harmonic on $\bbR^5$ i.e. it satisfies 
$$
\gamma_{\rho\rho}+\gamma_{zz}+\frac{3\gamma_\rho}{\rho}=0.
$$
This is a consequence of the fact that both $U$ and $U^{ref}$ satisfy Equation (\ref{Uharmonic}). We have
\begin{IEEEeqnarray*}{rCl}
\frac{\partial }{\partial \rho}\left(\frac{U_\rho}{\rho}\right)&=&\frac{U_{\rho\rho}}{\rho}-\frac{U_\rho}{\rho^2}, \\ \nonumber
\frac{\partial^2 }{\partial \rho^2}\left(\frac{U_\rho}{\rho}\right)&=&\frac{U_{\rho\rho\rho}}{\rho}-\frac{2U_{\rho\rho}}{\rho^2}+\frac{2U_\rho}{\rho^3}, \\ \nonumber
\frac{\partial^2 }{\partial z^2}\left(\frac{U_\rho}{\rho}\right)&=&\frac{U_{\rho zz}}{\rho}.
\end{IEEEeqnarray*}
Hence 
\begin{IEEEeqnarray*}{rCl}
\frac{\partial^2 }{\partial \rho^2}\left(\frac{U_\rho}{\rho}\right)+\frac{\partial^2 }{\partial z^2}\left(\frac{U_\rho}{\rho}\right)&=&\frac{U_{\rho\rho\rho}+U_{\rho zz}}{\rho}-\frac{2U_{\rho\rho}}{\rho^2}+\frac{2U_\rho}{\rho^3}\\
&=&\frac{\frac{\partial }{\partial \rho}\left(U_{\rho\rho}+U_{zz}\right)}{\rho}-\frac{2U_{\rho\rho}}{\rho^2}+\frac{2U_\rho}{\rho^3}\\
&=&-\frac{\frac{\partial  }{\partial \rho}\left(\frac{U_\rho}{\rho} \right)}{\rho}-\frac{2U_{\rho\rho}}{\rho^2}+\frac{2U_\rho}{\rho^3}\\
&=&-\frac{U_{\rho\rho}}{\rho^2}+\frac{U_\rho}{\rho^3}-\frac{2U_{\rho\rho}}{\rho^2}+\frac{2U_\rho}{\rho^3}\\
&=&-\frac{3U_{\rho\rho}}{\rho^2}+\frac{3U_\rho}{\rho^3}\\
&=&-3\frac{\partial }{\partial \rho}\left(\frac{U_\rho}{\rho}\right).\\
\end{IEEEeqnarray*}
Since the above equalities are simply a consequence of the fact $U$ and $U^{ref}$ both satisfy Equation (\ref{Uharmonic}). By linearity, the string of equalities holds when we replace $U$ by $U^{ref}$ and $\gamma$ satisfies the claimed PDE.

{
\item Next we shall show that $\gamma$ extends smoothly over $\partial \bbH.$  We know from the proof of Proposition (\ref{U}) that
$$
U(z,\rho)=S_0(z)+S_1(z)\log \rho +\rho S_2(z,\rho),
$$
where $S_0$ is a function of $z$, smooth in the segments $\cup_{i} ] z_i , z_{i+1} [$, $S_1(z)=2f(z)$, and $S_2$ is smooth in a neighbourhood of these segments in $\bbH.$ Similarly we have
$$
U^{ref}(z,\rho)=S^{ref}_0(z)+S^{ref}_1(z)\log(\rho)+\rho S^{ref}_2(z,\rho),
$$
where again $S^{ref}_1(z)=2f(z)$. Therefore
\begin{align*}
	\gamma(z,\rho) & =  \left(\frac{\partial S^{ref}_2}{\partial \rho} (z,\rho) - \frac{\partial S_2}{\partial \rho} (z,\rho) \right) \\
	& \ \ \ \ + \frac{S^{ref}_2(z,\rho)-S_2(z,\rho)}{\rho} .
\end{align*}
Hence, $\gamma$ will continuously extend to $\partial \mathbb{H}$ if
$$S^{ref}_2(z,\rho)-S_2(z,\rho) = O(\rho).$$ 
It turns out that both $S^{ref}_2(z,\rho)$ and $S_2(z,\rho)$ are $O(\rho)$ as it can be seen from the fact that both $U^{ref}$ and $U$ satisfy Equation (\ref{Uharmonic}). Indeed, using $f''(z)=0$ and our expansion near $\rho=0$, we find that
\begin{IEEEeqnarray*}{rCl}
	0&=&U_{\rho\rho}+U_{zz}+\frac{U_\rho}{\rho}\\
	&=&\frac{S_2(z,\rho)}{\rho} + S_0''(z) + 3 \frac{\partial S_2}{\partial \rho} + 2 \rho \frac{\partial^2 S_2}{\partial \rho^2},
\end{IEEEeqnarray*}
and recall that $S_2$ is smooth in a neighbourhood of the relevant segments in $\bbH.$ It therefore follows that we must have $S_2(z,\rho)=O(\rho)$ for the right hand side to remain bounded.
}
\end{itemize}

Consider the function
$$
\frac{U_\rho}{\rho}=\frac{\rho U_\rho}{\rho^2 }=\frac{1}{\rho^2\scal (g_K) } .
$$
From Lemma (\ref{lemma:scal>0}), $\scal (g_K)>0$ so that 
$$\frac{U\rho}{\rho}>0.$$ 
The rest of the arguments is also inspired on \cite{s}. We have
$$
\gamma\leq \frac{U^{ref}_\rho}{\rho}.
$$
On the other hand, from the explicit formula for $U^{ref}$ in Equation (\ref{explicitUBG}) we can see that 
$$
\frac{U^{ref}_\rho}{\rho}=\frac{2\left(A+{Bz}+\sum a_i \sqrt{(z-z_i)^2+\rho^2}\right)}{\rho^2}.
$$
Hence there is a constant $C$ so that
$$
\frac{U^{ref}_\rho}{\rho}\leq\frac{CR}{\rho^2},
$$
where $R=\sqrt{z^2+\rho^2}.$ As in \cite{s} this implies $\gamma$ is bounded. The argument is roughly as follows. Let $w$ be a point in $\bbR^5$. Because $\gamma$ is harmonic, 
$$
\gamma(w)=\frac{1}{aR^4}\int_{\partial B(w,R)}\gamma(w_0)dw_0,
$$
for a universal constant $a$, so that, using the bound on $\gamma$ we get
$$
\gamma(w)\leq\frac{C'}{R^3}\int_{\partial B(w,R)}\frac{dw_0}{\rho^2},
$$
for a constant $C'.$ As one can see from direct calculations 
$$
\int_{\partial B(0,R)}\frac{dw_0}{\rho^2}\leq C''R^2,
$$
and this shows $\gamma$ is bounded. Since it is smooth on $\bbR^5$ and harmonic, it must be constant, say $\gamma=k$ for $k \in \mathbb{R}$. Hence, from the equation
$$U_\rho^{ref}-U_\rho=k\rho,$$
we conclude that there is a function $K(z)$ such that 
$$
U=U^{ref}+\frac{k\rho^2}{2}-K(z).
$$
{However, from the fact that 
$$
\lim_{\rho\rightarrow 0}\frac{U^{ref}(z,\rho)}{\log(\rho^2)} = f(z) = \lim_{\rho\rightarrow 0}\frac{U(z,\rho)}{\log(\rho^2)} ,
$$
we conclude that $\frac{k\rho^2}{2}-K(z)$ must vanish and therefore $U=U^{ref}$ as we wanted to show.}
\end{proof}

Assume that $X=M\setminus D$ is endowed with two conformally K\"ahler, Ricci-flat, toric metrics $g$ and $\tilde{g}$ {whose adaptated K\"ahler forms are cohomologous. Now let $g_K$ and $\tilde{g}_K$ be extremal K\"ahler metrics conformal to $g$ and $\tilde{g}$ as in Proposition (\ref{U}). Together with the complex structure on $X$ each determines a K\"ahler form $\omega_K$ and $\tilde{\omega}_K$ and $[\omega_K]=[\tilde{\omega}_K].$ 
By Moser's trick there is an equivariant diffeomorphism $\Psi$ of $X$ such that $\Psi^*\omega_K=\tilde{\omega}_K.$  This diffeomorphism does not a priori preserve the complex structure. Because $(X,\omega_K)\simeq(X, \tilde{\omega}_K)$ the moment polytopes of the two manifolds are the same up to a translation and an $SL(2,\bbZ)$ transformation. Both metrics determine} functions $U$ and $\tilde{U}$ as in Proposition (\ref{U}) and the corresponding $f$ and $\tilde{f}$ defined via
$$
f(z)=\lim_{\rho\rightarrow 0}\frac{U(z,\rho)}{\log(\rho^2)}\quad \text{and } \tilde{f}(z)=\lim_{\rho\rightarrow 0}\frac{\tilde{U}(z,\rho)}{\log(\rho^2)}.
$$
{Now $f$ and $\tilde{f}$ are locally affine functions with the same slopes. We must have
$$
f(z)=\sum_{i=1}^{d-1}a_i|z-z_i|+A+Bz \quad  \text{and } \tilde{f}(z)=\sum_{i=1}^{d-1}a_i|z-\tilde{z}_i|+A+Bz.
$$
We will show that $z_i=\tilde{z}_i$ for $i=1,\cdots, d-1$ so that  $f$ and $\tilde{f}$ must coincide. 

\begin{lemma}
Let $X=M\setminus D$ be endowed with two conformally K\"ahler, Ricci-flat, toric metric $g$ and $\tilde{g}$  {whose adapted K\"ahler forms are cohomologous,} determining functions $U$ and $\tilde{U}$ as in Proposition (\ref{U}) and the corresponding $f$ and $\tilde{f}.$ Then $z_i=\tilde{z}_i$ so that $f=\tilde{f}.$
\end{lemma}
\begin{proof}

{Consider $E_i$, the $i$th edge on $P\setminus E.$ Its pre-image $S_i$ via the moment map in $X$ is a $\mathbb{CP}^1$ whose volume is given by
$$
\text{Vol}(S_i)=\int_{S_i}\omega=\int_{S_i}dx_1\wedge d\theta_1+dx_2\wedge d\theta_2.
$$
On $E_i,$ the direction in $\bbT^2$ which is not collapsed is perpendicular to $\nu_i=(\nu_i^1,\nu_i^2)$ so that we can write $(\theta_1,\theta_2)=t(\nu_i^2,-\nu_i^1)$ for $t\in ]0,2\pi[$ and 
$$
\text{Vol}(S_i)=\int_{S_i}\omega=2\pi\int_{E_i}\left(\nu_i^2dx_1-\nu_i^1dx_2\right).
$$
Now we want to express the integral above in $(z,\rho)$ coordinates, in which
$$
E_i=\{(z,\rho): \rho=0,\, z\in [z_{i-1},z_i]\}.
$$ 
Assume that $\scal(g_K)$ is not constant along $E_i$ which happens iff $\det(\fut,\nu_i)\ne 0.$ Then, the easiest way to express $(x_1,x_2)$ via $(z,\rho)$ is through the scalar curvature of $g_K.$ 
On $E_i,$ $\nu_i^1dx_1+\nu_i^2dx_2=0$ so that $$(dx_1,dx_2)=(\nu_i^2,-\nu_i^1)dx,$$ for a coordinate $x$ and
$$
\nu_i^2dx_1-\nu_i^1dx_2=|\nu_i|^2dx, \quad d\scal(g_K)=\det(\fut,\nu_i)dx.
$$
If $\det(\fut,\nu_i)\ne 0$ 
$$
\text{Vol}(S_i)=2\pi\frac{|\nu_i|^2}{\det(\fut,\nu_i)}\int_{E_i}d\scal(g_K).
$$
Near $E_i,$ 
$$
\rho U_{\rho}=\frac{1}{\scal(g_K)}=2f(z)+\rho S_2(z,\rho),
$$
so that 
\begin{align*}
\text{Vol}(S_i) & = 2\pi\frac{|\nu_i|^2}{\det(\fut,\nu_i)}\left(\scal(g_K)(z_i,0)-\scal(g_K)(z_{i-1},0)\right), \\
& = \frac{\pi|\nu_i|^2}{\det(\fut,\nu_i)}\left(\frac{1}{f(z_i)}-\frac{1}{f(z_{i-1})}\right).
\end{align*}
We consider a different moment coordinate, namely $\mu$, to follows the notation of \cite{bg} (Proposition 6.1 in \cite{bg}) given by
$$
\mu=-2\left(z+\frac{\rho H_\rho-2H}{H_z}\right)=2\left(\frac{\rho^2 U_z+2H}{\rho U_\rho}-z\right),
$$
where $H$ is the harmonic conjugate of $U$ that is $$H_z=\rho U_\rho, \quad H_\rho=-\rho U_z.$$
Near $E_i,$ the asymptotic behaviour of $U$ yields
$$
\begin{aligned}
&H_\rho=-\rho U_z=-2\det(\fut,\nu_i) \rho\log(\rho)+O(\rho\log(\rho)),\\ 
&H_z=\rho U_\rho=2\det(\fut,\nu_i)z+{2}C_i+O\left(1\right),\\
&H=\det(\fut,\nu_i)z^2+2C_iz+K_i+O(\rho),\\
&\mu=2\left(\frac{2z(\det(\fut,\nu_i)z+C_i)+2K_i}{2\det(\fut,\nu_i)z+C_i}-z\right)+O\left(\rho\right),
\end{aligned}
$$
{where $C_i$ are as in the proof of Lemma (\ref{nu1+nud}) and $K_i$ are constants}. This gives
$$
\mu=\frac{C_iz+2K_i}{f(z)},
$$
As before $(dx_1,dx_2)=(\nu_i^2,-\nu_i^1)dx,$ and we can assume $dx=Cd\mu,$ so
$$
\begin{aligned}
\text{Vol}(S_i)=&2C\pi{|\nu_i|^2}\int_{E_i}d\mu\\
&=2C\pi{|\nu_i|^2}\left(2K_i\left(\frac{1}{f(z_i)}-\frac{1}{f(z_{i-1})}\right)+C_i\left(\frac{z_i}{f(z_i)}-\frac{z_{i-1}}{f(z_{i-1})}\right)\right).
\end{aligned}
$$
When $\fut$ is not perpendicular to $E_i,$ we see that $\text{Vol}(S_i)$ determines
$$
\left(\frac{z_i}{f(z_i)}-\frac{z_{i-1}}{f(z_{i-1})}\right) \, \text{and} \left(\frac{1}{f(z_i)}-\frac{1}{f(z_{i-1})}\right).
$$
When $E_i$ is a compact edge in $P\setminus E,$ that is $1<i<d-2,$ then $\text{Vol}(S_i)$ is $2\pi$ times the euclidean length of $E_i$ so that this length determines $z_{i-1}$ and $z_i$ for $1<i<{d-1}$ when $E_i$ is not perpendicular to $\fut$. But there is at most one compact edge perpendicular to $\fut$ and in this case we can rely on adjacent edges to determine the missing $z_i.$}
\end{proof}

This Lemma finishes the proof of Theorem (\ref{thm_main}) as it implies that $U=U^{ref}$ and $\tilde{U}=U^{ref}$ so $U$ and $\tilde{U}$ also differ by a constant and the metrics are isometric.

\subsection{Proof of Theorem (\ref{thm_main2})}

{By assuming that the extremal vector field is not normal to either one of the non-compact edges in the moment polytope of $X$ we may conclude that $\det(\fut,\nu_1)\ne 0$ and $\det(\fut,\nu_{d-1})\ne 0.$ The conformally K\"ahler, Ricci-flat, toric metric determines a K\"ahler metric $g_K,$ a function $U$ as in Proposition (\ref{U}) and a function
$$
f(z)= \lim_{\rho\rightarrow 0}\frac{U(z,\rho)}{\log(\rho^2)}.
$$
From Lemma (\ref{nu1+nud}){ and Corollary (\ref{cor:B=0})}, we know that there are numbers $a_1,\cdots, a_d,$ a constant $A$ and $z_1,\cdots, z_{d-1},$ points on $\partial \bbH$ such that
$$
f(z)=A+\sum_{i=1}^{d-1}a_i|z-z_i|,
$$
where $2a_i=\det(\fut,\nu_{i+1}-\nu_{i})$ for $i=1,\cdots, d-1$. We have shown in Lemma (\ref{lemma:scal>0}) that $\scal(g_K)>0$ and this implies that $U_\rho>0$ since $\scal{g_K}=\frac{1}{\rho U_\rho}.$ We also know that the quantity
$$
V=-\left(\rho U_\rho+\frac{U^2_\rho U_{zz}}{U^2_{\rho z}+U^2_{zz}}\right),
$$
is positive so that $U_{zz}<0.$ A maximum principle for $U$ then implies that the function $$z\mapsto U(z,\rho)$$ is concave and $f$ is convex. These arguments appear in the proof of Lemma 4.2 in \cite{bg}. The convexity of $f$ implies the inequality $a_1,\cdots, a_d>0.$
As above set $U^{ref}$ to be 
\begin{equation}
\begin{aligned}
U^{ref}=&2\sum_{i=1}^{d-1} a_i\left(\sqrt{\rho^2+(z-z_i)^2}-(z-z_i)\log\frac{\sqrt{\rho^2+(z-z_i)^2}+z-z_i}{\rho}\right)\\
&+A\log(\rho^2).
\end{aligned}
\end{equation}
As in the proof of Theorem (\ref{thm_main}) above we can show that $U=U^{ref}.$ Now
$$
V=-\left(\rho U_\rho^{ref}+\frac{(U^{ref}_\rho)^2U^{ref}_{zz}}{(U^{ref}_{\rho\rho})^2+(U^{ref}_{zz})^2}\right),
$$
 is given by $1+\frac{2A}{R}$ as one can see by direct inspection. On the other hand, because $U=U^{ref},$ this must be positive which can only happen if $A>0.$ We are in the setting of Theorem (\ref{thmBG2}) and $U^{ref}=U^{BG}$ so that our metric coincides with a Biquard-Gauduchon metric. This finishes the proof of Theorem (\ref{thm_main2}).


\begin{thebibliography}{99}
  \bibitem[A]{a1}
  {\scshape M. Abreu}
   \emph {K\"ahler geometry of toric manifolds in symplectic coordinates}, Symplectic and contact topology: interactions and perspectives , Fields Inst. Commun., {\bf 35} (2003), Amer. Math. Soc., Providence, RI, 1--24.
   
  
  \bibitem[AS]{as}
    {\scshape M. Abreu, R. Sena-Dias}
    \emph{Scalar-flat K\"ahler metrics on non-compact symplectic toric 4-manifolds}, Ann. Global Anal. Geom. {\bf 41} (2012), no. 2, 209--239.
    
\bibitem[AG]{AG}
 	{\scshape A. Vestislav, P. Gauduchon}
	\emph{The Riemannian Goldberg–Sachs theorem}, International Journal of Mathematics {\bf 8} (1997), no. 4, 421-439.

  \bibitem[BG]{bg} 
   {\scshape O. Biquard, P. Gauduchon}
   \emph{ On Toric Hermitian ALF Gravitational Instantons},  Comm. Math. Phys. {\bf 399} (2023), {no. 1}, 389--422.
   
   \bibitem[CC]{cc}
{\scshape G. Chen, X. Chen}
\emph{Gravitational instantons with faster than quadratic curvature decay. I}, Acta Math. {\bf 227} (2021), (2), 263--307.


   \bibitem[CDS]{cds}
      {\scshape X. Chen, S. Donaldson, S. Sun}
       \emph{K\"ahler-Einstein metrics on Fano manifolds. I: Approximation of metrics with cone singularities}, J. Amer. Math. Soc. {\bf 28} (2015), {no. 1}, 183--197.
  
  \bibitem[CLW]{clw}
          {\scshape X. Chen, C. Lebrun, B. Weber}
  \emph{On conformally K\"ahler, Einstein manifolds}, J. Amer. Math. Soc. {\bf 21} (2008), {no. 4}, 137--1168.    

 \bibitem[D]{d}
{\scshape A. Derdzinski}
   \emph{Self-Dual K\"ahler Manifolds and Einstein Manifolds of Dimension Four}, Comp. Math. {\bf 49} (1983) 405--433. 

\bibitem[Do]{donjoyce}
{\scshape S. Donaldson}
\emph{A generalised Joyce construction for a family of nonlinear partial differential equations}, J. G\"okova Geom. Topol. GGT {\bf 3} (2009), 1--8.
   
   \bibitem[Do2]{do2}
{\scshape S. Donaldson}
   \emph{Extremal metrics on toric surfaces: a continuity method}, J. Differential Geom. {\bf 79} (2008), no. 3, 389--432.
   
\bibitem[L0]{l0} 
 {\scshape C. Lebrun} 
 \emph{Einstein metrics on complex surfaces}, Geometry and physics (Aarhus, 1995), Lecture Notes in Pure and Appl. Math {\bf 184}, 167--176.

\bibitem[L1]{l1}
{\scshape C. Lebrun} 
\emph{Bach-Flat K\"ahler Surfaces}, J. Geom. Analysis {\bf 30} (2020) 2491--2514.
      

%
      \bibitem[Li]{m}
{\scshape M. Li}
      \emph{Classification results for Hermitian non-K\"ahler gravitational instantons},  arXiv:2304.01609.

 \bibitem[Li2]{m2}
{\scshape M. Li}
      \emph{On 4-dimensional Ricci-flat ALE manifolds},  arXiv:2310.13197.

   \bibitem[OSD]{OSD}
{\scshape G. Oliveira, R. Sena-Dias}
\emph{ Minimal Lagrangian tori and action-angle coordinates}, Transactions of the American Mathematical Society, {\bf 374} (2021), 11, 7715--7742.


    \bibitem[S]{s}
    {\scshape R. Sena-Dias}
    \emph{ Uniqueness among scalar-flat K\"ahler metrics on non-compact toric 4-manifolds}, J. Lond. Math. Soc. {\bf 103} (2021), no. 2, 372--397.
   
   
   \bibitem[W]{w} 
   {\scshape  D.~Wright}
   \emph{The geometry of anti-self-dual orbifolds}, Ph.D thesis, Imperial College, London, 2009.

\end{thebibliography}
\end{document}